\documentclass[12pt]{article}
\usepackage{amsmath,amssymb,amsthm,epsfig}
\newtheorem{lemma}{Lemma}
\newtheorem{theorem}{Theorem}
\newcommand\hull{\mathop{\rm hull}}
\newcommand\rank{\mathop{\rm rank}}
\newcommand\range{\mathop{\rm Range}}
\newcommand\bz{{\bf 0}}
\renewcommand\b{{\bf b}}
\newcommand\e{{\bf e}}

\newcommand\g{{\bf g}}
\newcommand\h{{\bf h}}

\newcommand\q{{\bf q}}
\renewcommand\r{{\bf r}}
\newcommand\R{{\bf R}}
\newcommand\s{{\bf s}}
\renewcommand\t{{\bf t}}
\renewcommand\u{{\bf u}}
\renewcommand\v{{\bf v}}

\newcommand\x{{\bf x}}

\newcommand\eref[1]{$(\ref{#1})$}
\newcommand\sref[1]{Section~$\ref{#1}$}
\newcommand\lemref[1]{Lemma~$\ref{#1}$}

\newcommand\figref[1]{Fig.~$\ref{#1}$}
\title{On the complexity of nonnegative matrix factorization}
\author{Stephen A.~Vavasis\thanks{Department of Combinatorics and
Optimization,  University of Waterloo, 200 University Ave.~W., 
Waterloo, Ontario, N2L 3G1, Canada, {\tt vavasis@math.uwaterloo.ca}.
Supported in part by an NSERC award.}}
\begin{document}
\maketitle
\begin{abstract}
Nonnegative matrix factorization (NMF) has become a prominent technique for the
analysis of image databases, text databases and other information retrieval
and clustering applications.  In this report, we define an exact version
of NMF.  Then we establish several results about exact NMF:  (1) that it
is equivalent to a problem in polyhedral combinatorics; (2) that it
is NP-hard;  and (3) that a polynomial-time local search heuristic exists.
\end{abstract}
\section{Nonnegative matrix factorization}
Nonnegative matrix factorization (NMF) has emerged in the past decade
as a powerful tool for clustering data and finding features in
datasets.
Lee and Seung \cite{LeeSeung} showed that NMF can find features
in image databases, and Hofmann \cite{hofmann} showed that probabilistic
latent semantic analysis, a variant of NMF, can effectively
cluster documents according to their topics.
Cohen and Rothblum \cite{CohenRothblum} describe applications
for NMF in probability, quantum mechanics and other fields.

Nonnegative matrix factorization is defined as the following problem.
The input is $(A,k)$, where $A$ is an $m\times n$ matrix with nonnegative
entries, and $k$ is an integer such that $1\le k\le \min(m,n)$.
The output is a pair of matrices $(W,H)$ with
$W\in\R^{m\times k}$ and $H\in\R^{k\times n}$ such that $W$ and
$H$ both have nonnegative entries and such that
$A\approx WH$.  The precise sense in which $WH$ approximates $A$
may vary from one author to the next.  Furthermore, some authors
seek sparsity in either $W$ or $H$ or both.  Sparsity may be imposed
as a term in the objective function \cite{KimPark}.  

The algorithms proposed by 
\cite{hofmann,KimPark,LeeSeung}
and others for NMF have generally been based on  local
improvement heuristics.  Another class of heuristics is based on
greedy rank-one downdating \cite{greiner, bergman, ghodsi, gillis}.
No algorithm proposed in the literature comes with a guarantee of
optimality.  This suggests that solving NMF to optimality may be
a difficult problem, although to the best of our knowledge this
has never been established formally.  

The main purpose of this
paper is to provide the proof that NMF is NP-hard.  This paper 
considers a particular version of NMF that we call {\em exact NMF}, which
is defined as follows.  

\vspace{0.1in}
\underline{\bf EXACT NMF:} The input is 
a matrix $A\in\R^{m\times n}$ with nonnegative
entries whose rank is exactly $k$, $k\ge 1$.
The output is a pair of matrices $(W,H)$, where
$W\in\R^{m\times k}$ and $H\in\R^{k\times n}$,  $W$ and
$H$ both have nonnegative entries, and 
$A=WH$.  If no such $(W,H)$ exist, then the output is a statement
of nonexistence of a solution.  The decision version of 
EXACT NMF takes the same
input and gives as output
{\bf yes} if
such a $W$ and $H$ exists else it outputs {\bf no}.

\vspace{0.1in}
Implicit in the statement of exact NMF is an assumption
that the rank of $A$ is known.  If $A$ is specified as rational data,
then its rank may determined in polynomial time via reduction to
row-echelon form \cite{Edmonds}.  In practice, one would usually
prefer singular value decomposition to determine $\rank(A)$
\cite{GVL}.  

Observe that for any reasonable definition of the
approximation version of NMF (that is, the version described earlier
in which $\rank(A)$ is not constrained and in which
one requires $A\approx WH$ instead of exact equality), an optimal
algorithm when presented with an $A$ whose rank is exactly $k$ ought
to solve the exact NMF problem.  Thus, the ``standard'' NMF problem
using any norm is a generalization of EXACT NMF.
Therefore, any hardness result
that applies to exact NMF (such as our hardness result)
would presumably apply to most approximation
versions as well. 

A different generalization of EXACT NMF is the problem of nonnegative
rank determination due to Cohen and Rothblum, which asks,
given $A\in\R^{m\times n}$ with nonnegative entries, find
the minimum value of $k$ such that $A=WH$, $W\in\R^{m\times k}$,
$H\in\R^{k\times n}$, and $W,H$ have nonnegative entries.
Cohen and Rothblum give a super-exponential time algorithm
for this problem.  Since nonnegative rank
determination is a generalization of EXACT
NMF, our result shows that it
is also NP-hard.

The proof of NP-hardness of EXACT NMF
has two parts: In \sref{sec:poly} we show equivalence between EXACT NMF and
a problem in polyhedral combinatorics that we
call INTERMEDIATE SIMPLEX, and in \sref{sec:nphard} we
show the NP-hardness of this problem.  
A side result emerging from the proof of equivalence of EXACT NMF
to INTERMEDIATE
SIMPLEX is that
certain local-search heuristic for EXACT NMF can be solved with
linear programming (\sref{sec:local}).

\section{Equivalence to Intermediate Simplex}
\label{sec:poly}
In this section, we show an equivalence between 
the  EXACT NMF and
a problem in polyhedral combinatorics that we call INTERMEDIATE
SIMPLEX.  Although the focus in this section is on the decision
version of these problems, it is apparent from the proofs that
the search-versions could also be reduced to each other.  The reductions
use a number of arithmetic operations polynomial in $m$ and $n$ and
are therefore polynomial-time for both the usual Turing machine
model and the real-number model of Blum et al.~\cite{Blumetal}.

A problem related to INTERMEDIATE SIMPLEX 
was proposed by Cohen and Rothblum \cite{CohenRothblum} and
show to be
equivalent to nonnegative rank determination.  Therefore,
their results to some extent
imply the results of this section.  Nonetheless, we
present the equivalence here in order to provide detail for
our claim that all reductions are polynomial time.

The equivalence is shown in three steps by first showing an equivalence
to a problem  denoted P1.

\vspace{0.1in}
\underline{\bf P1}: Given matrices $W_0\in\R^{m\times k}$ and
$H_0\in\R^{k\times n}$ such that each has rank $k$ and such that
all entries of $W_0H_0$ are nonnegative, does there exist a 
nonsingular matrix $Q\in\R^{k\times k}$ such that
$W_0Q^{-1}$ and $QH_0$ both have all nonnegative entries?

\begin{theorem}
There is a polynomial-time reduction from EXACT NMF to P1 
and vice versa.
\end{theorem}

\begin{proof}
First we demonstrate the reduction of EXACT NMF to P1.
Suppose that we have an NMF instance, that is, a nonnegative
matrix $A$ of rank exactly $k$.  In polynomial time (using, e.g.,
reduction to row-echelon form) one can factor $A=W_0H_0$
such that $W_0\in\R^{m\times k}$ and $H_0\in\R^{k\times n}$.  (This
factorization does not solve exact NMF since the signs of the entries
of $W_0$ and $H_0$ are unknown.)  We claim that the original instance
of EXACT NMF is a yes-instance iff the instance of P1 is a 
yes-instance.  For one direction, suppose the instance of EXACT NMF
is a yes-instance, and
suppose $W,H$ are solutions
to exact NMF.  Then clearly $\range(A)=\range(W)=\range(W_0)$, which
is a dimension-$k$ subspace of $\R^n$, and similarly
$\range(A^T)=\range(H^T)=\range(H_0^T)$.  This means that there exist
two nonsingular $k\times k$ nonsingular matrices, say $P,Q$, such
that $W=W_0P$ and $H=QH_0$.  Thus, the equation $WH=W_0H_0$ may
be rewritten as $W_0PQH_0=W_0H_0$.  Notice that $W_0$ has a left
inverse and $H_0$ has a right-inverse since $W_0$ has full column
rank and $H_0$ has full row rank.  Thus, the previous equation
simplifies to $PQ=I$ (where $I$ denotes the $k\times k$ identity
matrix), i.e., $P=Q^{-1}$. Thus, $W_0Q^{-1}$ and $QH_0$ both have
nonnegative entries, so the instance of P1 is a yes-instance.
Conversely, suppose the instance of P1 is a yes-instance.  Then
there exists $Q$ such that $W=W_0Q^{-1}$ and $H=QH_0$ both have
all nonnegative entries, and $WH=W_0H_0=A$, so the instance of
exact NMF is a yes-instance.

For the opposite reduction, suppose we start with an instance
$(W_0,H_0)$ of P1.  Let $A=W_0H_0$; then $A$ is nonnegative
and has rank $k$.  We claim
that the instance of $A$ is a yes-instance if and only if the instance
of P1 is a yes-instance.  The proof uses essentially the same arguments
as in the previous paragraph.
\end{proof}

In order to simplify the main proof in this section, it is helpful
to define a slightly restricted version of P1 as follows by requiring
the last column of $W_0$ to be all 1's:

\vspace{0.1in}
\underline{\bf RESTRICTED P1}: 
Given matrices  $W_0\in\R^{m\times k}$ and
$H_0\in\R^{k\times n}$ such that (1) each has rank $k$;
(2) all entries of $W_0H_0$ are nonnegative; and (3)
the last column of $W_0$ is all 1's,
does there exists a 
nonsingular matrix $Q\in\R^{k\times k}$ such that
$W_0Q^{-1}$ and $QH_0$ both have all nonnegative entries?

\begin{theorem} There is a polynomial-time reduction from P1
to RESTRICTED P1 and vice versa.
\end{theorem}
\begin{proof}
Given an instance $(W_0,H_0)$ of P1, we can produce an instance of RESTRICTED
P1 as follows.  First, delete all rows of $W_0$ that are identically 0's.
This does not affect the rank of $W_0$, nor does it affect whether
the product $W_0H_0$ is nonnegative.  Finally, if $Q$ is a solution
problem P1 prior to deletion of identically zero rows, 
then it is still a solution
afterwards and vice versa.  

For the next step, let $\hat Q$ be a 
$k\times k$ nonsingular
matrix chosen such that $\hat QH_0\e=\e_k$.  
Here, $\e\in\R^n$ denotes the vector of all 1's,
 and $\e_k\in\R^k$ denotes
the last column of the $k\times k$ identity matrix.
Such a $\hat Q$ is guaranteed
to exist because $H_0\e$ cannot be zero:
$W_0H_0\e$ is the sum of columns of $W_0H_0$, which cannot be zero
since the columns of $W_0H_0$ are all nonnegative and $W_0H_0$ is
not identically zero by the assumption of rank at least 1.
Then observe that $(W_0\hat Q^{-1},\hat QH_0)$ is a yes-instance
of P1 iff $(W_0,H_0)$ is a yes-instance.
Such a $\hat Q$ may be found in polynomial time; for example,
any $k\times k$ nonsingular matrix whose last column is $H_0\e$ may
be taken as $\hat Q^{-1}$, and matrix inversion is polynomial-time
in the Turing machine model \cite{Edmonds}.   

Next, we observe that the last column of $W_0\hat Q^{-1}$ is
$W_0\hat Q^{-1}\e_k=W_0\hat Q^{-1}\hat QH_0\e=W_0H_0\e$.  We already
argued above that this vector is nonzero, but now we will argue
more strongly that every entry of $W_0H_0\e$ is positive.  
First, note that $W_0H_0\e$ 
is the sum of columns of the nonnegative matrix $W_0H_0$, and hence
all its entries are at least nonnegative.  Focus on entry $i$
of $W_0H_0\e$; since it is a sum of nonnegative terms, then if it were
zero then the entire $i$th row of $W_0H_0$ would have to be zeros.  
This means that the $i$th row of $W_0$ is orthogonal to every column
of $H_0$.  But since $H_0$ has full rank, this is possible only if
the $i$th row of $W_0$ is identically 0.  However, this possibility
is ruled out since we deleted identically zero rows of $W_0$.

Thus, the last column of $W_0\hat Q^{-1}$ contains all positive entries.
Therefore, we can consider the instance of P1 given by $(DW_0\hat Q^{-1},
\hat Q H_0)$ where $D$ is an $m\times m$ positive definite diagonal matrix
with diagonal entries chosen to make the last column of $DW_0\hat Q^{-1}$
equal to 1.  This instance of P1 is a yes-instance only if the original
instance was a yes-instance, because multiplying the first factor by
a positive definite diagonal matrix does not affect the signs of
$W_0H_0$ nor of $W_0\hat Q^{-1}Q^{-1}$.

The opposite reduction, namely the one from from RESTRICTED P1 to P1,
is trivial since any instance of RESTRICTED P1 is also an instance of P1.
\end{proof}

Now finally we get to the main new problem of this section.

\vspace{0.1in}
\underline{\bf INTERMEDIATE SIMPLEX:}  We are
given a polyhedron $P=\{\x\in\R^{k-1}: A\x\ge\b\}$
where $A\in \R^{n\times(k-1)}$ and $\b\in\R^n$
such
that $[A,\b]$ has rank $k$.  We are also given a set $S\subset \R^{k-1}$ of
$m$ points that are all contained in $P$ and that are not all
contained in any hyperplane (i.e., they
affinely span $\R^{k-1}$).  The question is whether
there exists a $(k-1)$-simplex $T$ such
that $S\subset T\subset P$.

\begin{theorem}
There is a polynomial-time reduction from RESTRICTED P1 to
INTERMEDIATE SIMPLEX and vice versa.
\end{theorem}

\begin{proof}
We will prove that both reductions exist 
at the same time by exhibiting a bijection
between instances of RESTRICTED P1 and instances of
INTERMEDIATE SIMPLEX such that both directions of the
bijection can be computed in polynomial time.

Given an instance $(W_0,H_0)$ of RESTRICTED P1, we produce an instance
of INTERMEDIATE SIMPLEX as follows.  The polytope
$P\subset \R^{k-1}$ is given by $\{\x\in\R^{k-1}:H_0(1:k-1,:)^T\x\ge -H_0(k,:)^T\}$. 
(This constraint may be written more compactly as $H_0^T[\x;1]\ge \bz$.)
The set $S$ of $m$ points in $P$ is given by 
$S=\{W_0(1,1:k-1)^T,\ldots, W_0(m,1:k-1)^T\}$.
The inverse mapping of this transformation starts with an instance
of INTERMEDIATE SIMPLEX given by 
$P=\{\x:A\x\ge \b\}$, $A\in\R^{m\times (k-1)}$
and $S=\{\x_1,\ldots,\x_m\}$ and produces
an instance of RESTRICTED P1 given 
by 
$$W_0=\left(
\begin{array}{cc}
\x_1^T  & 1 \\
\vdots & \vdots \\
\x_m^T & 1
\end{array}
\right)$$
and $H_0=[A^T;-\b^T]$.  

We first show that all side-constraints present in the statement
of RESTRICTED P1 and INTERMEDIATE SIMPLEX are satisfied.
The side-constraint that $[A,\b]$ has rank $k$ is equivalent (under this
bijection) to the side-constraint that $H_0$ has rank $k$.  
The side-constraint that $\x_1,\ldots,\x_m$ affinely span $\R^{k-1}$
is equivalent to requiring that $[\x_1;1],\ldots,[\x_m;1]$ 
linearly span $\R^k$, i.e., to the side-constraint that $W_0$ has
rank $k$.  Finally, the side constraint that $S\subset P$
means that $A\x_i\ge\b$ for $i=1,\ldots, m$, i.e., $[A,-\b][\x_i;1]\ge 0$,
which is hence equivalent to the side-constraint that all entries of
$W_0H_0$ are nonnegative.

We now show that the above bijection
in both directions  maps yes-instances to yes-instances.
Let $(S,P)$ be an instance of INTERMEDIATE SIMPLEX and $(W_0,H_0)$
the corresponding instance of RESTRICTED P1. 
Let $T$ be a putative solution to the instance of INTERMEDIATE
SIMPLEX.  Let its
vertices be $\g_1,\ldots,\g_k$, which are vectors in $\R^{k-1}$.
The condition that $T\subset P$ is equivalent
to requiring 
$\g_1,\ldots,\g_k\in P$, i.e., to
$H_0^T[\g_i;1]\ge \bz$ for each $i=1,\ldots, k$.  
If we let 
\begin{equation}
G=\left(
\begin{array}{ccc}
\g_1 & \cdots & \g_k \\
1 & \cdots & 1
\end{array}
\right),
\label{eq:gdef}
\end{equation}
then we have shown that the condition $T\subset P$ is equivalent to
requiring $H_0^TG$ has all nonnegative entries.

The condition that $S\subset T$ means that for all $i=1,\ldots,m$,
$\x_i\in T$.  Recall that, by definition,
 a vector is inside a simplex if
it is a convex combination of its vertices.
Let $\q_i$ be the putative vector of coefficients of the convex combination
that expresses $\x_i$ in the hull of the vertices of $T$, for
$i=1,\ldots,m$.  In other words, 
\begin{equation}
[\g_1,\ldots,\g_k]\q_i=\x_i,
\label{eq:convhull}
\end{equation}
plus the requirements  that the entries of $\q_i$ are nonnegative and sum to 1.
The latter constraint may be combined with \eref{eq:convhull} to write
$G\q_i=[\x_i;1]$ where $G$ is as in \eref{eq:gdef}, i.e., 
$\q_i=G^{-1}W_0(i,:)^T$.  The hypothesis that $S\subset T$ is thus
equivalent to
the condition that each entry of $G^{-1}W_0^T$
for each $i=1,\ldots,m$ is nonnegative, i.e., 
all entries of
$G^{-1}W_0^T$ must be nonnegative.   Hence, we have shown that 
$T$ is a solution to the instance $(S,P)$ if and only if
$G^T$ is a solution to the instance $(W_0,H_0)$ of RESTRICTED P1.

The argument is essentially the same in the other direction.
Given an instance $(W_0,H_0)$ of RESTRICTED P1, let 
$(S,P)$ be the corresponding instance of INTERMEDIATE SIMPLEX.
Let $Q$ be a putative solution to the RESTRICTED P1 instance,
and let $\g_1,\ldots,\g_k$ be the columns of $Q^T$.  Using the
arguments in the previous paragraph shows that $W_0Q^{-1}$
and $QH_0$ have nonnegative entries iff $S\subset T$ and $T\subset P$.
\end{proof}

An easy consequence of our transformation of EXACT NMF to INTERMEDIATE
SIMPLEX is the observation that when $\rank(A)=2$, the NMF instance
is always a yes-instance.  The reason is that the resulting instance
of INTERMEDIATE SIMPLEX is 1-dimensional in which case $P$ is an
interval.  However, if $P$ is an interval then it is already a simplex,
so one could take $T=P$ to solve the instance.  This observation
yields a simple linear-time algorithm to find an exact nonnegative
factorization of $A$  case $\rank(A)=2$.
case.
This result was first established by Cohen and Rothblum \cite{CohenRothblum},
who also propose a simple linear-time algorithm.

\section{INTERMEDIATE SIMPLEX is NP-hard}
\label{sec:nphard}
In this section, we will argue that the problem INTERMEDIATE
SIMPLEX introduced in the previous section
is NP-hard.  

Before delving into the statement of the main theorem and its 
proof, we first state the following simpler lemma and 
proof.  This lemma describes the `gadget' used in the main theorem
below to encode a setting of a boolean variable.
\begin{lemma}
Consider the following instance of INTERMEDIATE SIMPLEX: the
polyhedron $P$ is given by $P=\{(x,y)\in\R^2: 0\le x,y\le 1\}$, while
the set $S$ is given by $\{(0,1/2),(1,1/2),(1/2,1/4),(1/2,3/4)\}$.
This instance has precisely two solutions $T_0$ or $T_1$ defined by
$T_0=\hull\{(0,0),(0,1),(1,1/2)\}$ and
$T_1=\hull\{(1,0),(1,1),(0,1/2)\}$.
\label{lem:intsimp}
\end{lemma}

\begin{figure}
\begin{center}
\epsfig{file=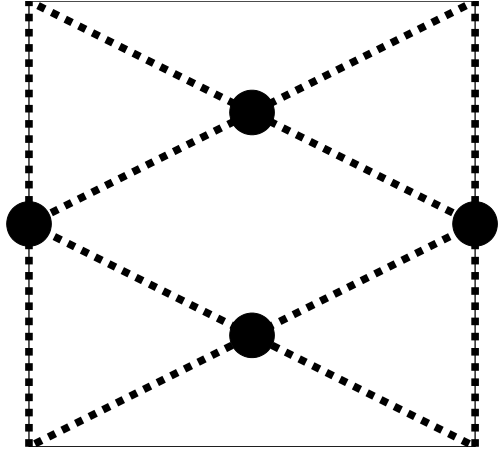,height=6cm}
\caption{Illustration of \protect\lemref{lem:intsimp}.  The
four large dots are the points in $S$; the thin solid line is
the boundary of $P$, and the two triangles indicated with thick
dashed lines are the two possible solutions $T_0$ and $T_1$.}
\label{fig:lemma}
\end{center}
\end{figure}

A diagram of the lemma is given in \figref{fig:lemma}.
It is easy to check that the side-constraints 
of INTERMEDIATE SIMPLEX (that $S\subset P$, that
$[A,\b]$ has full column rank, that $S$ affinely spans $\R^2$)
are satisfied by the above
instance.

\begin{proof}
The fact that $S\subset T_i\subset P$, $i=0,1$, is elementary to
check. The fact that there are no other solutions is proved as follows.
Suppose $T$ is a solution.
Let $E_0$ and $E_1$ be the two parallel edges of $P$ given by
$E_0=\{0\}\times [0,1]$ and $E_1=\{1\}\times[0,1]$.
Observe that the point $(0,1/2)\in S$ lies on $E_0$,
which means that the face of $T$ containing $(0,1/2)$ must be
either 0 or 1-dimensional, and if it is 1-dimensional then it must
be a subset of $E_0$.  Similarly, $(1,1/2)$ must lie on a 0- or 1-dimensional
boundary of $T$.  It is not possible for both $(0,1/2)$ and $(1,1/2)$ to
lie on 1-dimensional boundaries since a triangle cannot have two parallel
edges.  It is also not possible for both $(0,1/2)$ and $(1,1/2)$ to be
0-dimensional boundaries because in this case $[0,1]\times\{1/2\}$ would
be a bounding segment of $T$.  Then all of $T$ would have to
be either above or below the segment, but then $T$ would fail to cover
either $(1/2,1/4)$ or $(1/2,3/4)$, points in $S$.  Thus, the only
possibilities are (1) that $(0,1/2)$ is a vertex of $T$, and $T$ has
an edge that is a subset of $E_1$, or (2) that $(1,1/2)$ is a vertex of
$T$, and $T$ has an edge that is a subset of $E_0$.  But now one checks that
in either case, in order to cover the two points $(1/2,3/4)$ and $(1/2,1/4)$,
the entire edge $E_0$ or $E_1$ must be taken as an edge of $T$.
\end{proof}

We now turn to the main result for this section, namely, the
NP-hardness of INTERMEDIATE SIMPLEX.
In particular, we reduce 3-SAT \cite{GJ} to this problem.  Our reduction
uses integers whose magnitude is polynomial in the instance of the
3-SAT instance, and hence our result is `strong' NP-hardness.
Recall that an instance of 3-SAT involves $p$ boolean variables
denoted $x_1,\ldots, x_p$ and $q$ clauses denoted $c_1,\ldots,c_q$.
Each clause is a disjunction of three literals, where a literal
is either a variable $x_j$ or its complement $\tilde x_j$.  
An instance of 3-SAT is a yes-instance if and only if there exists
a setting of the variables, that is, an assignment of a value of
either 0 or 1 to each variable, such that each clause is
satisfied, i.e., at least one of its three literals is 1.
It is assumed that the same variable does not occur twice 
(either in complemented or plain form) in any particular clause.

Given such an instance of 3-SAT, we define the following instance
of INTERMEDIATE SIMPLEX.  It contains $3p+q$ variables
(i.e., $k-1=3p+q$) denoted $s_i,t_i,u_i$, $i=1,\ldots, p$, and
$v_j$, $j=1,\ldots, q$.  These variables are written
as $(\s,\t,\u,\v)$ for short.
The polyhedron $P$ is defined by
the following inequalities:
\begin{equation}
P=\left\{
\begin{array}{lll}
(\s,\t,\u,\v): & 
\bz \le\s\le\u, \\
& \bz\le\t\le \u, \\
& \bz\le\u\le \e, \\
& \v\ge \bz, \\
& s_i-2t_i \le v_j &\mbox{whenever $\tilde x_i\in c_j$}, \\
& 2t_i-2s_i-u_i \le v_j &\mbox{whenever $ x_i\in c_j$}
\end{array}
\right\}.
\label{eq:Pdef}
\end{equation}
Here, $\e$ denotes the vector of all 1's.  In the above
usage, $\e\in\R^p$, but we shall also use $\e$ to denote
the vector of all 1's in $\R^q$.

Let $\e_i$ denote the $i$th column of the identity matrix (either
the $p\times p$ or $q\times q$ identity).
The set of points $S$ is defined as
follows.  Each of the points in the following
equation is also given a name for future reference.
\begin{equation}
S=\left\{
\begin{array}{lll}
\bz, \\
(\e/(4p), \e/(4p),\e/(2p), 2.5\e/(8p)) &(\equiv \b), \\
(\bz,\bz,\bz,\e_j) & (\equiv \h_j), & j=1,\ldots, q \\
(\bz, \e_i/4, \e_i/2,\e) &(\equiv \r_i^1), & i=1,\ldots, p \\
(\e_i/2, \e_i/4, \e_i/2,\e)  &(\equiv \r_i^2), & i=1,\ldots, p\\
(\e_i/4, \e_i/8, \e_i/2,\e)&(\equiv \r_i^3), & i=1,\ldots, p,  \\
(\e_i/4, 3\e_i/8, \e_i/2,\e) &(\equiv \r_i^4), & i=1,\ldots, p 
\end{array}
\right\}.
\label{eq:Sdef}
\end{equation}

Let us first confirm that the side-constraints of INTERMEDIATE
SIMPLEX are satisfied by this instance.  Since $\bz\in S$, 
$S$ affinely spans $\R^{3p+q}$ iff it linearly spans $\R^{3p+q}$.
Points $\h_j$, $j=1,\ldots,q$,
span the subspace defined by
the last $q$ coordinate entries.
Fix some $i\in\{1,\ldots,p\}$.
Subtract
$\h_1+\ldots+\h_q$ from the three points
$\r_i^1,\r_i^2,\r_i^3$.
This yields three points whose nonzero entries are restricted to the
$(s_i,t_i,u_i)$ positions; in these positions the three points have
coordinate entries $(0,1/4,1/2)$, $(1/2,1/4,1/2)$ and $(1/4,1/8,1/2)$,
which are linearly independent.  Thus, the subspace
indexed by $(s_i,t_i,u_i)$ is spanned by $S$.  This is true for all $i$,
so therefore the points in $S$ span all of $\R^{3p+q}$.

The next side-constraint is that the linear inequalities defining
$P$ are independent.  One checks that the constraints $\s\ge \bz$,
$\t\ge \bz$, $\u\ge\bz$, $\v\ge \bz$ imply that the constraint
matrix contains
a $(3p+q)\times (3p+q)$ identity matrix and hence has independent columns.
We can also check that the right-hand side is independent of the
columns of the matrix; if it were dependent, then there would be
a point such that all the constraints are active at that point, which
is obviously impossible (e.g., the constraints $\u\ge\bz$ and
$\u\le \e$ cannot be simultaneously active).
The final side-constraint is that $S\subset P$, which is an
elementary matter to check.

The main theorem of this section is as follows.
\begin{theorem}
The instance of 3-SAT is a yes-instance if and only if the above
instance of INTERMEDIATE SIMPLEX is a yes-instance.  In other
words, the 3-SAT instance has a satisfying assignment if and only
if there exists a simplex $T$ such that $S\subset T\subset P$.
\end{theorem}

\begin{proof}
First, let us choose some terminology for the coordinates of $\R^{3p+q}$.
The individual coordinates may be denoted by $s_i$, $t_i$, $u_i$ or
$v_j$ for $i=1,\ldots,p$ and $j=1,\ldots,q$.  Collectively, the
three coordinates $(s_i,t_i,u_i)$ are called the ``$x_i$ coordinates''
since they correspond to the $i$th boolean variable in the 3-SAT
instance.

Let $T$ be a solution to the instance of INTERMEDIATE SIMPLEX.
From $T$ we will construct a satisfying assignment $\sigma$ for the
3-SAT instance.
Clearly $T$ has exactly $3p+q+1$ vertices.
Observe first that the point $\bz$ is an extreme point of $P$ and also
lies in $S$, and therefore one vertex of $T$ must be $\bz$. 

Similarly, observe that each $\h_j$, $j=1,\ldots, q$,
lies on extreme edge of $P$, and therefore $T$ must have $q$ vertices
of the form $\lambda_j\h_j$, $j=1,\ldots, q$
with each  $\lambda_j\ge 1$.  

This accounts for all but $3p$ of the vertices of $T$.  
For an $i\in\{1,\ldots,p\}$, let us say that a vector 
$(\s,\t,\u,\v)\in\R^{3p+q}$ is {\em $x_i$-supported} if
it is zero in all the $x_j$-coordinates for all $j\in\{1,\ldots,p\}-\{i\}$.  
More
strongly, say that it is {\em $x_i$-positive} if it 
is $x_i$-supported and is positive in at least one of the 
$x_i$ coordinates.
Fix a particular $i\in\{1,\ldots,p\}$ and consider the four $S$-points
$\r_i^1,\ldots,\r_i^4$ which are all $x_i$-positive.
Projected into the $x_i$ coordinates, these points 
are $(0,1/4,1/2)$, $(1/2,1/4,1/2)$, $(1/4,1/8,1/2)$ and 
$(1/4,3/8,1/2)$.
Since none of the $T$-vertices has negative entries, each of 
$\r_i^1,\ldots,\r_i^4$
must lie in the hull only of $T$-vertices that are
$x_i$-supported such as
$\bz,\lambda_1\h_1,\ldots,\lambda_q\h_q$.  Furthermore, it must lie in the hull
of at least one $x_i$-positive vertex of $T$.
In fact, there must be at least three such $x_i$-positive $T$-vertices
since the four points, when projected into the $x_i$ coordinates,
are linearly independent.  Thus, $T$ must have at least three
$x_i$-positive vertices for each $i=1,\ldots,p$.  
Since there are only $3p$ vertices of $T$ not
yet enumerated, we conclude that $T$ must have exactly three 
$x_i$-positive vertices for each $i$, which we denote
$\g_{i,1},\g_{i,2},\g_{i,3}$.

Let $\bar\g_{i,1},\bar\g_{i,2},\bar\g_{i,3}\in\R^3$
denote the 
$x_i$ coordinates of $\g_{i,1},\g_{i,2},\g_{i,3}$.
By the assumption that $T$ covers the four points
$(0,1/4,1/2)$, $(1/2,1/4,1/2)$, $(1/4,1/8,1/2)$ and $(1/4,3/8,1/2)$
in the projection into the $x_i$ coordinates, we 
conclude that there must exist a $3\times 4$ matrix $B$ with
nonnegative entries such that
$$\left(\bar\g_{i,1},\bar\g_{i,2},\bar\g_{i,3}\right)\cdot B
=\left(
\begin{array}{cccc}
0 & 1/2 & 1/4 & 1/4 \\
1/4 & 1/4 & 1/8 & 3/8 \\
1/2 & 1/2 & 1/2 & 1/2
\end{array}
\right).$$
As mentioned above, all of 
$\bar\g_{i,1},\bar\g_{i,2},\bar\g_{i,3}$ are nonzero.
Because of the inequalities $\bz\le\s\le\u$ and $\bz\le\t\le\u$ that define
$P$, it must be the case that the third entries of
$\bar\g_{i,1},\bar\g_{i,2},\bar\g_{i,3}$ are all positive and no smaller
than the first and second entries.  Therefore, define new
vectors $\hat \g_{i,1},\hat \g_{i,2},\hat \g_{i,3}$ that are
all exactly 1/2 in the last coordinate and have other coordinates
lying in $[0,1/2]$ obtained by rescaling each of
$\bar\g_{i,1},\bar\g_{i,2},\bar\g_{i,3}$ by twice its third
coordinate.  By rescaling $B$ in a reciprocal manner, we
find that there is a nonnegative matrix $\hat B$ such that
$$\left(\hat\g_{i,1},\hat\g_{i,2},\hat\g_{i,3}\right)\cdot \hat B
=\left(
\begin{array}{cccc}
0 & 1/2 & 1/4 & 1/4 \\
1/4 & 1/4 & 1/8 & 3/8 \\
1/2 & 1/2 & 1/2 & 1/2
\end{array}
\right).$$
By consider the third row of the above system of equations, we conclude
that each column of $\hat B$ sums to exactly 1.  Then dropping the
third row on both sides yields the equation
$$\left(\hat\g_{i,1}(1:2),\hat\g_{i,2}(1:2),\hat\g_{i,3}(1:2)\right)\cdot \hat B
=\left(
\begin{array}{cccc}
0 & 1/2 & 1/4 & 1/4 \\
1/4 & 1/4 & 1/8 & 3/8 
\end{array}
\right),$$
where the notation $\v(1:2)$ denotes the first two entries of a vector.
Now we observe that this is precisely a half-sized version of
the instance of INTERMEDIATE
SIMPLEX described in the preliminary lemma of this section, namely, find
three points lying in $[0,1/2]^2$ whose convex hull covers the four points
$\{(0,1/4),(1/2,1/4),(1/4,1/8),(1/4,3/8)\}$.  As established
by the lemma, there are precisely two
solutions to this system, which we will
denote $T_0/2$ and $T_1/2$.
Let $C_0$ be the set of $i$'s such that
the triangle defined by 
$(\hat\g_{i,1}(1:2),\hat\g_{i,2}(1:2),\hat\g_{i,3}(1:2))$
is $T_0/2$, while $C_1$ is the set of $i$'s such that this triangle
is $T_1/2$.  
Thus we conclude that for $i\in C_0$,
\begin{equation}
(\bar\g_{i,1},\bar\g_{i,2},\bar\g_{i,3})=
(\mu_{i,1}(0,0,1),\mu_{i,2}(0,1,1),\mu_{i,3}(1,1/2,1)),
\label{eq:gc0}
\end{equation}
and for $i\in C_1$,
\begin{equation}
(\bar\g_{i,1},\bar\g_{i,2},\bar\g_{i,3})=
(\mu_{i,1}(1,0,1),\mu_{i,2}(1,1,1),\mu_{i,3}(0,1/2,1)),
\label{eq:gc1}
\end{equation}
where $\mu_{i,k}>0$ for $k=1,2,3$.
This determines the $x_i$ entries of $\g_{i,k}$, $k=1,2,3$, and
the remaining $x_j$ entries are zeros since $\g_{i,k}$ is $x_i$-positive.
Therefore, it remains only to determine the
$v_j$ entries of $\g_{i,k}$, $k=1,2,3$.  There are several constraints
on these entries as follows.  First, we have the inequalities
$v_j\ge 0$, and thus all those entries must be nonnegative.
Next, we have the constraints 
$s_i-2t_i \le v_j$ whenever $\tilde x_i\in c_j$ and
$2t_i-2s_i-u_i \le v_j$ whenever $x_i\in c_j$.
These inequalities are redundant whenever their left-hand side is
nonpositive since we have already constrained $v_j\ge 0$. 
Thus, we need only consider the cases when the left-hand sides are
positive.  We
see that the left-hand side of the first inequality
$s_i-2t_i \le v_j$ 
is positive only in the case of
$\bar\g_{i,1}$ only for $i\in C_1$, and the left-hand side
of the second inequality
$2t_i-2s_i-u_i \le v_j$ 
is positive only in the case of
$\bar\g_{i,2}$ only for $i\in C_0$.  
Thus, for $i\in C_1$, 
for all $j$
such that $\tilde x_i$ occurs as a literal in clause $c_j$,
we must have 
\begin{equation}
\g_{i,1}|_{v_j}\ge \mu_{i,1}.
\label{eq:gc1vj}
\end{equation}
(Here, the notation $\g_{i,1}|_{v_j}$ denotes the $v_j$ coordinate
entry of $\g_{i,1}$.)
Similarly, for $i\in C_0$, 
for all $j$ such that $x_i$ occurs as a literal in clause $c_j$,
we must have 
\begin{equation}
\g_{i,2}|_{v_j}\ge\mu_{i,2}.
\label{eq:gc0vj}
\end{equation}

Next, $T$ must contain the point $\b$ from \eref{eq:Sdef}, so
there must be coefficients
$\alpha_{i,k}$, $i=1,\ldots,p$, $k=1,2,3$ and $\theta_j$, $j=1,\ldots,q$ adding
up to at most 1 and all nonnegative such that
\begin{equation}
\b=\sum_{i=1}^p\sum_{k=1}^3\alpha_{i,k}\g_{i,k}+
\sum_{j=1}^q \theta_j\lambda_j\h_j.
\label{eq:bsum}
\end{equation}
Fix a particular $i$.  The projection of $\b$
into $x_i$ coordinates
is $\bar\b=(1/(4p),1/(4p),1/(2p))$.  Referring back
to \eref{eq:gc0} and \eref{eq:gc1}, one can see that
regardless of whether $i\in C_0$ or
$i\in C_1$, $\bar\b$ is
expressed uniquely as $\bar\b=\bar\g_{i,1}/(8p\mu_{i,1})+
\bar\g_{i,2}/(8p\mu_{i,2})+\bar\g_{i,3}/(4p\mu_{i,3})$.
Therefore, 
\begin{equation}
\alpha_{i,1}=1/(8p\mu_{i,1});\quad\alpha_{i,2}=1/(8p\mu_{i,2});\quad
\alpha_{i,3}=1/(4p\mu_{i,3}).
\label{eq:alpha}
\end{equation}
Suppose $i\in C_0$.  Then for each $j$ such that $x_i$ occurs
as a literal in clause $c_j$, if we combine  \eref{eq:gc0vj} and
\eref{eq:alpha}, we obtain 
$$
\left.\sum_{k=1}^3\alpha_{i,k}\g_{i,k}\right|_{v_j}\ge 1/(8p).
$$
The identical inequality holds when $i\in C_1$ and
$\tilde x_i\in c_j$.

Now, sum the preceding inequality for $i=1,\ldots,p$ to obtain
\begin{equation}
\left.\sum_{i=1}^p
\sum_{k=1}^3\alpha_{i,k}\g_{i,k}\right|_{v_j}\ge m_j/(8p),
\label{eq:bsumvj}
\end{equation}
where $m_j$ is the number of literals $x_i\in c_j$ 
with $i\in C_0$ plus the number of literals $\tilde x_i\in c_j$ 
with $i\in C_1$.
Let us now combine these inequalities:
From \eref{eq:Sdef}, $\b|_{v_j}=2.5/(8p)$.
From \eref{eq:bsum}, 
$$\left.\b|_{v_j}\ge \sum_{i=1}^p
\sum_{k=1}^3\alpha_{i,k}\g_{i,k}\right|_{v_j},$$
since the last term of \eref{eq:bsum} is nonnegative.
Finally, from \eref{eq:bsumvj}, the above summation is 
at least $m_j/(8p)$.  Thus, we conclude that $m_j\le 2.5$.  Since
$m_j$ is integral, this means $m_j\le 2$.  
Let $\sigma$
be the setting of the $x_i$'s in the 3-SAT
instance defined by taking $x_i=1$ for 
$i\in C_1$ and $x_i=0$ for $i\in C_0$.  
Then if $x_i\in c_j$ and $i\in C_0$, this literal is falsified
in the clause.
Similarly, if $\tilde x_i\in c_j$ and $i\in C_1$, then this
literal is also falsified.  In other words, $m_j$ is the number
of literals in clause $c_j$ falsified by assignment $\sigma$.
We have just argued that $m_j\le 2$ for all
$j=1,\ldots,q$. In other words, for each clause, there are most two
literals falsified by assignment $\sigma$.  Therefore, $\sigma$
is a satisfying assignment for the 3-SAT instance.

Summarizing, we have proved that if there is a 
simplex $T$ solving the instance of 
INTERMEDIATE SIMPLEX, then there are exactly three vertices of $T$ that are
$x_i$-positive for each $i=1,\ldots,p$; that, based on these vertices,
$i$ can be classified as either $C_0$ or $C_1$; and that
the assignment $\sigma$ of the boolean variables in the original
3-SAT instance derived from $C_0$ and $C_1$ must be a satisfying
assignment.

Conversely, suppose the 3-SAT instance has a satisfying assignment.
The vertices of $T$ will be $\bz,\lambda_1\h_1,\ldots,\lambda_q\h_q$
together with $\g_{i,1},\g_{i,2},\g_{i,3}$ for each $i=1,\ldots,p$,
defined as follows.
Let $C_0$ index the variables set to 0 by the satisfying assignment and
$C_1$ the variables set to 1.  Define 
$\bar\g_{i,1},\bar\g_{i,2},\bar\g_{i,3}$ as in \eref{eq:gc0}
and \eref{eq:gc1} according
to $C_0$ and $C_1$.  Take
$\mu_{i,k}=5/8$ for all $(i,k)$.  (Any other value slightly greater
than $1/2$ will work.)
When $i\in C_0$ and $x_i$ is a literal in $c_j$, then take
$\g_{i,2}|_{v_j}=5/8$.  When $i\in C_1$ and $\tilde x_i$ is a literal
in $c_j$, then take $\g_{i,1}|_{v_j}=5/8$.  In all other cases,
take $\g_{i,k}|_{v_j}=0$.  It is easy to see that all the inequalities
defining $P$ are satisfied by these choices.  Furthermore, all the
points in $S$ are covered by convex combinations of the $3p+q+1$ points
$\bz,\lambda_1\h_1,\ldots,\lambda_q\h_q,\g_{1,1},\ldots,\g_{p,3}$,
which are the vertices of $T$.

For example, the point
$\r_i^1=(\bz, \e_i/4, \e_i/2,\e)$ in the case that $i\in C_0$ is
expressed as $(2/5)\g_{i,1}+(2/5)\g_{i,2}+\h$, where $\h$ is some
linear combination of $\lambda_1\h_1,\ldots,\lambda_q\h_q$ chosen to make the
$v_j$ entries each equal to 1.  (Note that the $v_j$ entries
of $(2/5)\g_{i,1}+(2/5)\g_{i,2}$ before $\h$ is added will be
either 0 or $1/4$).  The total sum of the coefficients to express
$(\bz, \e_i/4, \e_i/2,\e)$ is $2/5+2/5+h_1$, where $h_1$ is
the sum of the coefficients needed in the terms of $\h$.  
Select $\lambda_1,\ldots,\lambda_q$  to be large
scalars so that
we can be assured that $4/5+h_1\le 1$.  If this sum
is less than 1, then we include a contribution of $\bz$, another
vertex of $T$, in the linear combination to make the sum of
coefficients exactly 1.

Similarly, as sketched out earlier, to obtain the point
$\b=(\e/(4p), \e/(4p),\e/(2p), 2.5\e/(8p))$ in the hull
of the vertices of $T$, we use \eref{eq:bsum} with coefficients
chosen according to \eref{eq:alpha}.
This choice of $\alpha_{i,j}$'s yields $x_i$ coordinate entries
equal to $(1/(4p),1/(4p),1/(2p))$ for each $i$ and
has entries less
than or equal to $2/(8p)$ in each $v_j$ coordinate entry.  Then, as
above, one can include additional terms involving $\bz$ and 
$\lambda_1\h_1,\ldots,\lambda_q\h_q$
to complete the convex combination.  One point to note is that the sum
of the $\alpha_{i,k}$ coefficients appearing in \eref{eq:bsum}, assuming
$\mu_{i,k}=5/8$, is equal to $4/5$, and hence does not exceed 1.
Addition of 
the $\theta_j$ coefficients will make the total higher but still
less than 1 provided $\lambda_1,\ldots,\lambda_q$ are all chosen
to be very large.
\end{proof}

\section{Local-search heuristics}
\label{sec:local}
In this section we will prove a theorem about the INTERMEDIATE
SIMPLEX problem that will suggest a class of local-search heuristics.
The theorem is as follows.

\begin{theorem}
Consider an instance of INTERMEDIATE SIMPLEX given by polytope 
$P\subset\R^{k-1}$ with $n$ facets
and point set $S\subset P$ with $m$ vectors.  
Suppose there exists a solution
$T$, and suppose that all vertices of $T$
are given except for one.  Then the set of feasible positions for the
last vertex is defined by a system of linear equations and inequalities
($mk$ equalities and $n+mk$ inequalities).
\end{theorem}

\begin{proof}
Let the vertices of $T$ be denoted $\v_1,\ldots,\v_k$, and suppose
all are known except $\v_k$.  Two sets of constraints must be
satisfied, namely, those arising from the requirement $S\subset T$
and those arising from $T\subset P$.  Since the simplex $T$ is
assumed to be a solution, the given values of $\v_1,\ldots,\v_{k-1}$
must all lie in $P$, and hence the constraint on $\v_k$ to ensure
that $T\subset P$ is simply that $\v_k\in P$.  This clearly amounts
to a set of $n$ linear inequalities that must be satisfied 
by $P$.  

Next, consider the requirement $S\subset T$; choose a particular
vector $\b\in S$.  If $\b$ is in the hull of $\v_1,\ldots,\v_{k-1}$
then $\b$ is in $T$ no matter what choice is made for $\v_k$, so such
a $\b$ does not impose any constraint on $\v_k$.  Else suppose
$\b$ is not in the hull of $\v_1,\ldots,\v_{k-1}$.  Then the requirement
on $\v_{k}$ is that  there exist $\lambda_1,\ldots,\lambda_k$
such that $\lambda_1\v_1+\cdots+\lambda_k\v_{k}=\b$, 
$\lambda_i\ge 0$, $i=1,\ldots,k$, and 
$\lambda_1+\cdots+\lambda_k=1$.  This constraint is nonlinear because
of the product of unknowns $\lambda_k\v_k$.  However, we can rearrange
it into a linear constraint by dividing through by $\lambda_k$
(which is nonzero by the hypothesis that $\b$ is not in the hull
of $\v_1,\ldots,\v_{k-1}$) and defining new variables 
$\alpha_i=\lambda_i/\lambda_k$, $i=1,\ldots,k-1$, and
$\alpha^*=1/\lambda_k$.  Then the above constraints become
$\alpha_1\v_1+\cdots+\alpha_{k-1}\v_{k-1}+\v_k=\alpha^*\b$, 
$\alpha_i\ge 0$, $i=1\ldots,k-1$, $\alpha^*\ge 0$, 
$\alpha_1+\cdots+\alpha_{k-1}+1=\alpha^*$, which are all linear.
There are $k$ equality constraints and $k$ inequality constraints in
this system.  A system of this kind is needed for each point in $S$.        
\end{proof}

The preceding theorem suggests a local search heuristic for INTERMEDIATE
SIMPLEX.
One can choose as an initial guess $T$ a large simplex that contains
all of $S$ but perhaps is not contained in $P$.  Then one adjusts
the vertices of $T$ one at a time, optimizing a criterion that
minimizes departure of the vertex from feasibility.  Because the
feasible positions for the vertex under consideration form a polyhedron,
several possible criteria such as 2-norm distance to feasibility
would constitute convex programming problems.  Thus, on each iteration
of the local search algorithm, one could reposition a single vertex
of $T$ optimally until a solution is found.

\section{Acknowledgement}
The author acknowledges helpful discussions about this
material with Levent Tuncel
and Ali Ghodsi (University
of Waterloo), Jon Kleinberg  and  Lillian Lee (Cornell),
Chris Ding (Lawrence Berkeley Laboratory) and Vincent Blondel
(Louvain).  

\bibliography{../../Bibfiles/nips,../../Bibfiles/nips1}
\bibliographystyle{plain}
\end{document}